\documentclass[12pt]{amsart}
\usepackage{amsmath,amssymb,euscript,mathrsfs, eufrak}
\usepackage{hyperref}
\usepackage[all]{xy}
\usepackage{enumerate}

\pushQED{\qed}

\renewcommand{\setminus}{\smallsetminus}
\renewcommand{\phi}{\varphi}

\newcommand{\R}{\mathbb{R}}

\newcommand{\bk}{\mathbf{k}}

\def\<{\ensuremath{\langle}}
\def\>{\ensuremath{\rangle}}

\newtheorem{theorem}{Theorem}[section]
\newtheorem{lemma}[theorem]{Lemma}

\theoremstyle{definition}

\newcommand{\excise}[1]{}

\begin{document}

\title{A non-Abelian analogue of Whitney's $2$-isomorphism theorem}
\author{Eric Katz}
\address{Department of Combinatorics \& Optimization, University of Waterloo, 200 University Avenue West, Waterloo, ON, Canada N2L 3G1} \email{eekatz@math.uwaterloo.ca}

\date{}
\thanks{}

\begin{abstract}
We give a non-abelian analogue of Whitney's $2$-isomorphism theorem for graphs.  Whitney's theorem states that the cycle space determines a graph up to $2$-isomorphism.  Instead of considering the cycle space of a graph which is an abelian object, we consider a mildly non-abelian object, the $2$-truncation of the group algebra of the fundamental group of the graph considered as a subalgebra of  the $2$-truncation of the group algebra of the free group on the edges.  The analogue of Whitney's theorem is that this is a complete invariant of $2$-edge connected graphs:
let $G,G'$ be $2$-edge connected finite graphs; if there is a bijective correspondence between the edges of $G$ and $G'$ that induces equality on the $2$-truncations of the group algebras of the fundamental groups, then $G$ and $G'$ are isomorphic.
\end{abstract}

\maketitle

\section{Introduction}\label{s:intro}

Let $G$ be a finite graph, and let $\bk$ be a field.  Pick an orientation for every edge of $G$.  The chain group $C_1(G;\bk)$ is the vector space over $\bk$ generated by the edges of $G$.  Within this space is the cycle space $Z_1(G;\bk)$, the vector space generated by cycles in $G$.  Whitney's $2$-isomorphism theorem \cite{Whitney}, \cite[Section 5.3]{Oxley} states that $Z_1(G;\bk)\subset C_1(G;\bk)$ determines $G$ up to two moves: vertex cleaving and Whitney twists.  Specifically, it states that if $G'$ is another finite graph and $\phi:\overrightarrow{E}(G)\rightarrow \overrightarrow{E}(G')$ is bijective map of oriented edges (that is, if $e^{-1}$ denotes $e$ with the opposite orientation then $\phi(e^{-1})=(\phi(e))^{-1}$), if the induced map on chain groups satisfies $\phi_*(Z_1(G;\bk))=Z_1(G';\bk)$ then after performing some combination of these moves on $G'$, one has that $\phi$ is a graph isomorphism.  Two graphs related after these moves are said to be $2$-isomorphic.  Because these moves cannot be applied non-trivially to a $3$-connected graph, it implies that $3$-connected graphs $G$,$G'$ satisfying the hypotheses are isomorphic.

It is natural to ask if there is a modification of this theorem that in certain situations allows one to conclude that more general $G$ and $G'$ are isomorphic without performing any moves.  Our approach is to consider invariants of graphs that are finer than the cycle space.  Our  invariants interpolate between the cycle space which is abelian and homological in nature and the fundamental group which is non-abelian and homotopy-theoretic.  The invariants are labelled by a positive integer $k$.  In an intuitive sense, our invariants are refinements of the cycle space.  The cycle space contains information about which edges are in a cycle but nothing about their order.  Our $k$th invariant contains information about which edges are in a cycle but also for a cycle and a list of $k$ edges, $e_{i_1},e_{i_2},\dots,e_{i_k}$, it contains information about how many times $e_{i_1},e_{i_2},\dots,e_{i_k}$ occur in that order in the cycle (counted with signs and multiplicities).  It turns out that for $k\geq 2$, our invariant is a complete invariant for $2$-edge connected graphs.  

We now give the definition of our invariant deferring some notation and background about group algebras to the next section.  Let $v_0$ be a vertex of $G$ that we will call the base-point.  Let $\pi_1(G,v_0)$ be the fundamental group of $G$ with base-point $v_0$.   Pick arbitrary orientations on the edges of the graph.  Each closed path $G$ based at $v_0$ can be expressed as a word in the edges $e_{i_1}^{\pm 1}e_{i_2}^{\pm 1}\dots e_{i_n}^{\pm 1}$ where the path consists of the edges $e_{i_1},e_{i_2},\dots,e_{i_n}$ traversed in order and the sign of the exponent is determined by whether or not the edge occurs with its given orientation in $G$.  This association of words to closed paths gives a homomorphism 
\[w:\pi_1(G,v_0)\rightarrow F_{E(G)}\]
where $F_{E(G)}$ is the free group on the edges.  

One may take truncated group algebras of the groups involved to obtain finite dimensional $\bk$-algebras.  Consider the induced maps of group algebras,
\[w_*:\bk[\pi_1(G,v_0)]\rightarrow\bk[F_{E(G)}].\]
The group algebras are equipped with augmentation homomorphisms (see Section \ref{notation}),
\[\varepsilon_\pi:\bk[\pi_1(G,v_0)]\rightarrow\bk,\ \varepsilon_F:\bk[F_{E(G)}]\rightarrow\bk\]
with kernels $J_{\pi_1(G,v_0)}, J_{F_{E(G)}}$, respectively.  
For any non-negative integer $k$, $\bk[\pi_1(G,v_0)]/J_{\pi_1(G,v_0)}^{k+1}$, $\bk[F_{E(G)}]/J_{F_{E(G)}}^{k+1}$ are finite-dimensional $\bk$-algebras, called the $k$-truncations.
The map $w_*$ descends to a map of truncated group algebras:
\[w_*:\bk[\pi_1(G,v_0)]/J_{\pi_1(G,v_0)}^{k+1}\rightarrow \bk[F_{E(G)}]/J_{F_{E(G)}}^{k+1}.\]
The $2$-truncation is sufficient to give a complete invariant of $2$-edge connected finite graphs according to our main theorem:

\begin{theorem} \label{t:theorem} Let $G,G'$ be $2$-edge connected finite  graphs.
Let $\phi:\overrightarrow{E}(G)\rightarrow \overrightarrow{E}(G')$ be a bijective map of oriented edges.   If we have the following equality of subalgebras in $\bk[F_{E(G')}]/J_{F_{E(G')}}^3$:
\[\phi_*(w_*(\bk[\pi_1(G,v_0)]/J_{\pi_1(G,v_0)}^3))=w'_*(\bk[\pi_1(G',v'_0)]/J_{\pi_1(G',v'_0)}^3)\]
 then $\phi$ is a graph isomorphism satisfying $\phi(v_0)=v'_0$.
\end{theorem}

We note that this theorem is similar to Whitney's theorem.  The hypotheses of the theorem imply the analogous fact about a lower  order truncation of group algebras:
\[\phi_*(w_*(\bk[\pi_1(G,v_0)]/J_{\pi_1(G,v_0)}^2))=w'_*(\bk[\pi_1(G',v'_0)]/J_{\pi_1(G',v'_0)}^2).\]
As we will discuss below, $\bk[\pi_1(G,v_0)]/J_{\pi_1(G,v_0)}^2\cong\bk\oplus Z_1(G;\bk)$.  On the other hand, we have $\bk[F_{E(G)}]/J_{F_{E(G)}}^2\cong\bk\oplus C_1(G;\bk)$.  Consequently, the hypotheses for the $1$-truncation imply that $G$ and $G'$ have the same cycle space, and Whitney's theorem states that they are $2$-isomorphic.

The non-abelian structure that is used in this paper is only very mildly noncommutative.  In fact, the fundamental group acts unipotently on its truncated group algebra. We could perhaps also call this result the unipotent analogue of Whitney's theorem.  One may ask if there are more places in combinatorics where one can incorporate noncommutativity to prove rigidity theorems.  The author hopes, perhaps overly speculatively, that there are similar results that make up a combinatorial theory analogous to Grothendieck's anabelian program in algebraic geometry \cite{G}.

This work, we hope, hints at an extension of the notion of matroids.  Can one axiomatize the map 
\[w_*:\bk[\pi_1(G,v_0)]/J_{\pi_1(G,v_0)}^{3}\rightarrow \bk[F_{E(G)}]/J_{F_{E(G)}}^{3}\]
the way one axiomatizes the cycle space of a graph into a matroid?  This suggests a sort of unipotent matroid.  Which ordinary matroids lift to unipotent matroids?  If not all, is there a combinatorial characterization of the obstruction to a unipotent structure?

This paper was inspired by Hain's theory of the mixed Hodge structure on the fundamental group of complex manifolds \cite{Hain},\cite[Ch. 9]{PSMixed} which follows Morgan's work on the mixed Hodge structure on the truncation of the group ring of the fundamental group \cite{Morgan}.  Hain is able to put a Hodge structure on a unipotent completion of the fundamental group.  As an application, Hain \cite{Hain} and Pulte \cite{Pulte} give a pointed Torelli theorem for Riemann surfaces, that is, they show that the mixed Hodge structure on the truncated group ring is a complete invariant of a pointed Riemann surface (up to some finite ambiguity for the base-point).  The recent work of Caporaso and Viviani \cite{CV} proves a Torelli theorem for graphs and tropical curves by making use of Whitney's $2$-isomorphism theorem, so it seemed natural to ask if there is an extension of Whitney's theorem that could be used to prove a pointed Torelli theorem for graphs and tropical curves.   

We would like to thank Jim Geelen, Richard Hain, David Jackson, Farbod Shokrieh, David Wagner, Thomas Zaslavsky, and David Zureick-Brown for valuable discussions.

\section{Truncated group algebras of fundamental groups} \label{notation}
To a group $\Gamma$ and a field $\bk$, one can associate the group algebra $\bk[\Gamma]$.  The reader loses nothing by taking $\bk$ to be $\R$. This is the algebra over $\bk$ whose elements are formal linear combinations of the form $\sum a_gg$ where $a_g\in\bk$ is $0$ for all but finitely many elements.   Multiplication in the group algebra is the linear extension of $g\cdot g'=gg'.$  Therefore, when $\Gamma=F_n$, the free group on $n$ generators, $k[F_n]$ is the free noncommutative polynomial algebra on $n$ indeterminates.
Let the augmentation map $\varepsilon:\bk[\Gamma]\rightarrow\bk$ be the linear extension of $\varepsilon:g\mapsto 1$ for all $g\in G$.  Let $J=\ker(\varepsilon)$ be the augmentation ideal.  It is the set of all elements of the form $\sum a_g g$ where $\sum a_g=0$.  The $k$-truncation of the group algebra is $\bk[\pi]/J^{k+1}$.

The assignment of group algebras to a group is functorial, so the homomorphism $w$ (described above) induces a homomorphism of $\bk$-algebras,
\[w_*:\bk[\pi_1(G,v_0)]\rightarrow \bk[F_{E(G)}].\]
Moreover, one has an induced map of truncations:
\[w_*:\bk[\pi_1(G,v_0)]/J_{\pi_1(G,v_0)}^{k+1}\rightarrow \bk[F_{E(G)}]/J_{F_E(G)}^{k+1}.\]
Now we can consider functoriality under graph morphisms.  Let $\phi:G\rightarrow G'$ be a morphism of graphs, that is a map $\phi:V(G)\cup E(G)\rightarrow V(G')\cup E(G')$ such that $\phi(V(G))\subseteq V(G')$ and for every $v\in V(G)$, $e\in E(G)$ with $v\in e$ either $\phi(v)=\phi(e)$ or $\phi(e)\in E(G')$ and $\phi(v)\in \phi(e)$.  Given two directed graphs $G,G'$ with base-points $v_0,v'_0$ and a morphism $\phi:G\rightarrow G'$ satisfying $\phi(v_0)=v'_0$, we have an induced map of fundamental groups $\phi_*:\pi_1(G,v_0)\rightarrow\pi_1(G',v'_0)$.  Moreover, if $\phi$ is a graph morphism (not necessarily satisfying $\phi(v_0)=v_0'$), there is an induced map $\phi_*:F_{E(G)}\rightarrow F_{E(G')}$ defined as follows: if $\phi(e)\in V(G')$ then $\phi_*(e)=\emptyset$, the empty word; if $\phi_*(e)\in E(G)$ then $\phi_*(e)=\phi(e)^{\pm 1}$ where the sign of the exponent depends on whether $\phi$ is orientation preserving or reversing on the edge $e$.  Consequently, we have the following commutative diagram of truncated group algebras:
\[\xymatrix{
\bk[\pi_1(G,v_0)]/J_{\pi_1(G,v_0)}^{k+1}\ar[d]_{\phi_*}\ar[r]^>>>>>{w}&\bk[F_{E(G)}]/J_{F_{E(G)}}^{k+1}\ar[d]^{\phi_*}\\
\bk[\pi_1(G',v'_0)]/J_{\pi_1(G',v'_0)}^{k+1}\ar[r]^>>>>{w'}&\bk[F_{E(G')}]/J_{F_{E(G')}}^{k+1}
}\]

In this paper, we will work with the $1$- and $2$-truncations.  In our situation, every group will be a free group.  In these cases, the description of the truncated group algebra is rather straightforward.  Let $F_n$ be the free group on generators $x_1,\dots,x_n$.  We note that the there is a short exact sequence of vector spaces
\[\xymatrix{
0\ar[r]&J_{F_n}^k/J_{F_n}^{k+1}\ar[r]&\bk[F_n]/J_{F_n}^{k+1}\ar[r]&\bk[F_n]/J_{F_n}^{k}\ar[r]&0
}\]
Because $\bk[F_n]/J_{F_n}^{k}$ is a vector space,
the above exact sequence splits.  
Therefore, we have the vector space isomorphism,
\[\bk[F_n]/J_{F_n}^{k+1}\cong \bk\oplus J_{F_n}^1/J_{F_n}^{2}\oplus\dots\oplus J_{F_n}^k/J_{F_n}^{k+1}.\]
Now, $J_{F_n}^k/J_{F_n}^{k+1}$ is generated as a vector space by polynomials of the form 
\[(x_{i_1}-1)(x_{i_2}-1)\dots(x_{i_k}-1)\]
where $1$ is the element corresponding to the empty word $\emptyset$.
  Consequently,  the vector space $\bk[F_n]/J_{F_n}^2$ is generated by $1$ together with the following basis of $J_{F_n}/J_{F_n}^2$: 
 \[(x_1-1),(x_2-1),\dots,(x_n-1).\]
 The element $1$ acts as the identity, and the multiplication of two elements of $J_{F_n}$ is always $0$.  The natural map $F_n \rightarrow \bk[F_n]/J_{F_n}^2$ takes $x_i$ to $1+(x_i-1)$ and $x_i^{-1}$ to $1-(x_i-1)$.  Consequently, the word 
$x_{i_1}^{b_1}\dots x_{i_l}^{b_l}$ is mapped to $1+b_1(x_{i_1}-1)+\dots+b_{i_l}(x_{i_l}-1)$.  It follows that $\bk[F_n]/J_{F_n}^2\cong \bk\oplus (F_n^{\operatorname{ab}}\otimes\bk)$ where $F_n^{\operatorname{ab}}$ is the abelianization of $F_n$.  Consequently, we have that the truncation of $w_*$
\[w_*:\bk[\pi_1(G,v_0)]/J_{\pi_1(G,v_0)}^2\rightarrow \bk[F_{E(G)}]/J_{F_{E(G)}}^2\]
is isomorphic to 
\[w_*:\bk\oplus H_1(G;\bk)=\bk\oplus Z_1(G;\bk)\rightarrow\bk\oplus C_1(G;\bk)\]
and therefore contains the description of the cycle space.

The $2$-truncation is  richer.  It has a vector space basis given by 
\[1,(x_i-1),(x_{i}-1)(x_{j}-1)\]
 as $i$ and $j$ range from $1$ to $n$.  The natural map from $F_n$ takes $x_i$ to $1+(x_i-1)$ and $x_i^{-1}$ to $1-(x_i-1)+(x_i-1)^2$.  Consequently, we may write for $b_i=\pm 1$, 
\[x_i^{b_i}\mapsto 1+b_i(x_i-1)+\delta_{b_i,-1}(x_i-1)^2\]
where $\delta_{i,j}$ the Kronecker delta,
and we may conclude  that a word of the form $x_{i_1}^{b_1}\dots x_{i_l}^{b_l}$ where $b_i=\pm 1$ is mapped as follows:
\[x_{i_1}^{b_1}\dots x_{i_l}^{b_l}\mapsto 1+\sum_j b_j (x_{i_j}-1)+\sum_{j<k} b_jb_k(x_{i_j}-1)(x_{i_k}-1)+\sum_{j|b_j=-1} (x_{i_j}-1)^2.\]
Note that this counts with signs the number of times $x_j$ comes before $x_k$ in a word.



If $\phi:\overrightarrow{E}(G)\rightarrow \overrightarrow{E}(G')$ is a bijective map of oriented edges (with 
no requirement on the incidence of the edges), then it induces a homomorphism $\phi: F_{E(G)}\rightarrow F_{E(G')}$.
In the case where $\phi$ gives a graph isomorphism of a subgraph $H\subset G$ onto its image, we will denote the restriction of $\phi$ to $H$ by $\phi|_{H}$.  In this case, it makes sense to speak of the value of $\phi|_{H}$ on vertices of $H$.

\section{Proof of Theorem}

The proof will be by induction on the number of edges for which $\phi$ is a graph isomorphism.  Our main tool is the following lemma:

\begin{lemma} \label{l:isoextension}
Let $(G,v_0),(G',v'_0)$ be $2$-edge connected finite rooted graphs and let $\phi:\overrightarrow{E}(G)\rightarrow \overrightarrow{E}(G')$ be a bijective map of oriented edges such that we have the following equality of subalgebras in $\bk[F_{E(G')}]/J_{F_{E(G')}}^3$:
\[\phi_*(w_*(\bk[\pi_1(G,v_0)]/J_{\pi_1(G,v_0)}^3))=w'_*(\bk[\pi_1(G',v'_0)]/J_{\pi_1(G',v'_0)}^3)\]
Let $\gamma$ be a closed path in $G$ based at $v_0$.  Let $e$ be an edge that is not a self-edge at $v_0$ and that occurs exactly once in $\gamma$ so that $\gamma=\gamma_- e \gamma_+$ for paths $\gamma_-,\gamma_+$.  Suppose $\phi$ is a graph isomorphism of the path $\gamma_-$ onto its image.  Then the terminal vertex of the path $\phi|_{\gamma_-}(\gamma_-)$ is the equal to the initial vertex of $\phi(e)$.   Moreover, $\phi|_{\gamma_-}(v_0)=v'_0$.
\end{lemma}

\begin{proof}

Write $e=xy$ where we may have $x=y$.  We may suppose that $e$ occurs with its given orientation in $\gamma$.  Express $w_*(\gamma)\in \bk[F_{E(G)}]/J_{F_{E(G)}}^3$ in terms of the basis $
 1,(e_i-1),(e_i-1)(e_j-1)$ for an enumeration $\{e_i\}$ of the edges in $G$.
 The coefficient of $(e-1)$ is $1$ because $e$ occurs once in $\gamma$. 
 Group together the terms in $w_*(\gamma)$ of the form $c_j(e_j-1)(e-1)$ for varying $j$ as $\eta (e-1)$.  Therefore, $\eta$ corresponds to the edges coming before $e$ in $\gamma$ and hence is the $1$-chain representing $\gamma_-$.  Consequently, 
 if $\partial:C_1(G;\bk)\rightarrow C_0(G;\bk)$ is the differential in simplicial homology, then $\partial \eta=\partial \gamma_-=x-v_0$.

Let $e'=\phi(e)$.  Write $e'=x'y'$.  We must show that $\phi|_{\gamma_-}(x)=x'$.  By hypothesis, we can find an equality in $\bk[F_{E(G')}]/J_{F_{E(G')}}^3$ of the form
\[ \phi_*(w_*(\gamma))=\sum_i a_iw'_*(\delta'_i)\]
for closed loops $\delta'_i$ based at $v'_0$ and $a_i\in\bk$.  
Now we decompose the based loop $\delta'_i$ into paths according to each occurrence of $e'$ as follows:
\[\delta'_i=(\alpha_{i,1})(e')^{b_{i,1}}(\alpha_{i,2})(e')^{b_{i,2}}\dots(\alpha_{i,l})(e')^{b_{i,l_i}}(\alpha_{i,l_i+1})\]
where $b_{i,j}=\pm 1$ and the $\alpha_{i,j}$'s do not involve $e'$.  
Because the coefficient of $(e-1)$ in $w_*(\gamma)$ is $1$, the coefficient of $(e'-1)$ in $\phi_*(w_*(\gamma))$ is also equal to $1$.  This implies that we have
\[\sum_i\sum _{j=1}^{l_i} a_ib_{i,j}=1.\]
 We group together terms of the form $c'_j(e'_j-1)(e'-1)$ for varying $j$ in $w'_*(\delta_i)$ to get $\eta_i'(e'-1)$ where we view $\eta'_i$ as a chain in $C_1(G';\bk)$.  As a chain, $\eta_i$ has the following expression:
\begin{eqnarray*}
\eta_i&=&\sum_{\substack{j\in\{1,\dots,l_i\}\\b_{i,j}=1}} \left(\alpha_{i,1}+b_{i,1}e'+\alpha_{i,2}+b_{i,2}e'+\dots+ \alpha_{i,j}\right)\\
&&-\sum_{\substack{j\in\{1,\dots,l_i\}\\b_{i,j}=-1}}\left(\alpha_{i,1}+b_{i,1}e'+\alpha_{i,2}+b_{i,2}e'+\dots+ \alpha_{i,j}-e'\right)
\end{eqnarray*}
 Each term in parentheses is the chain of a path from $v'_0$ to $x'$.
Consequently, we have the following value for the differential $\partial:C_1(G';\bk)\rightarrow C_0(G';\bk)$:
\[\partial \eta'_i=\sum_{j=1}^{l_i} b_{i,j}(x'-v'_0)\]
Then $\eta'=\sum_{i} a_i\eta'_i$ satisfies
\[\partial \eta'=\sum_i \sum_{j=1}^{l_i} a_ib_{i,j}(x'-v'_0)=x'-v'_0.\]
Now, by the description of the homomorphism $\phi_*$ on $\bk[F_{E(G)}]/J_{F_{E(G)}}^3$, we have that 
\[\phi_*(\eta)=\eta'.\]  Taking the differential of both sides of the above equality, we get
\[\phi|_{\gamma_-}(x)-\phi|_{\gamma_-}(v_0)=x'-v'_0\]
which implies that $\phi|_{\gamma_-}(x)=x'$ and $\phi|_{\gamma_-}(v_0)=v'_0$.
\end{proof}

We now give the proof of the theorem:

\begin{proof}
We first prove the case where neither $G$ and $G'$ have any self-edges at their base-points.
We induct on the size of connected subgraphs $H\subseteq G$ for which $\phi|_H$ is an isomorphism onto its image.  We begin the induction with $H=\{v_0\}$.  

For the inductive step, if $H$ is not all of $G$, let $e$ be an edge of $E(G)\setminus E(H)$ that is incident to a vertex $x$ of $H$.  This is possible because $G$ is connected.   Write $e=xy$.  Let $\gamma_-$ be a path in $H$ from $v_0$ to $x$. 

Because $e$ is not a cut edge, there is a path $\gamma_+$ in $G$ from $y$ to $v_0$ that avoids $e$.  
Now, one can apply Lemma \ref{l:isoextension} to $\gamma=\gamma_-e\gamma_+$ and conclude that the terminal point of $\phi|_{H}(\gamma_-)$ is equal to the initial point of $\phi(e)$.  

We must show that $\phi$ extends to an isomorphism of $H\cup\{e\}$ onto its image.  First, consider the case that $y$ is a vertex of $H$.  Let $\delta_-$ be a path in $H$ from $v_0$ to $y$ avoiding $e$.  By applying Lemma \ref{l:isoextension} to $\delta_-e^{-1}(\gamma_-)^{-1}$, we get the terminal point of $\phi(e)$ is $\phi|_{H}(y)$.     Consequently,  $\phi$ extends to an isomorphism of $H\cup\{e\}$ onto its image.  Now, consider the case where $y$ is not a vertex of $H$.   We must show that the terminal point of $\phi(e)$ is not a vertex of $\phi|_{H}(H)$.  If it was, one could apply the above argument to the map $\phi^{-1}$ on $\phi(H)$ and conclude that the terminal point of $e$ is a vertex of $H$.  This contradiction completes the proof.

For the general case, let $H$ and $H'$ be the complement of the self-edges at the base-points in $G$ and $G'$, respectively.  By building up $H$ edge-by-edge as above we get that $\phi$ is an isomorphism between $H$ and $H'$. Since $G$ and $G'$ have the same cycle spaces, they have the same cyclotomic numbers.  Therefore, we can conclude that they have the same number of self-edges at their base-points and that they are isomorphic.  
\end{proof}

\bibliographystyle{amsplain}

\begin{thebibliography}{99}

\bibitem{CV}
Lucia Caporaso and Filippo Viviani.
\emph{Torelli theorem for graphs and tropical curves.}
Duke Math. J. \textbf{153} (2010), no. 1, 129--171. 

\bibitem{G}
Alexandre Grothendieck. 
\emph{Esquisse d'un programme.} \emph{Geometric Galois Actions, 1}, 5--48, Cambridge Univ. Press, Cambridge, 1997. 

\bibitem{Hain}
Richard Hain.
\emph{The geometry of the mixed Hodge structure on the fundamental group.}
\emph{Algebraic geometry}, Bowdoin, 1985 (Brunswick, Maine, 1985), 247--282, Amer. Math. Soc., 1987.

\bibitem{Morgan}
John Morgan.
\emph{The algebraic topology of smooth algebraic varieties.}
Inst. Hautes \'{E}tudes Sci. Publ. Math. \textbf{48} (1978), 137--204. 

\bibitem{Oxley}
James Oxley.
\emph{Matroid theory.} 
Second edition. Oxford Graduate Texts in Mathematics, 21. Oxford University Press, Oxford, 2011. 

\bibitem{Pulte}
Michael Pulte.
\emph{The fundamental group of a Riemann surface: mixed Hodge structures and algebraic cycles.} 
Duke Math. J. \textbf{57} (1988), no. 3, 721--760.

\bibitem{PSMixed}
Chris Peters and Joseph Steenbrink. \emph{Mixed {H}odge structures}. Springer-Verlag, Berlin, 2008.




\bibitem{Whitney}
Hassler Whitney. 
\emph{2-isomorphic graphs.}
Amer. J. Math. \textbf{55} (1933), no. 1-4, 245--254. 

\end{thebibliography}
\def\cprime{$'$}
\providecommand{\bysame}{\leavevmode\hbox to3em{\hrulefill}\thinspace}
\providecommand{\MR}{\relax\ifhmode\unskip\space\fi MR }
\providecommand{\MRhref}[2]{%
  \href{http://www.ams.org/mathscinet-getitem?mr=#1}{#2}
}
\providecommand{\href}[2]{#2}

\end{document}